\documentclass{amsart}
\usepackage{amsmath,amsthm,amsfonts,amssymb,latexsym,mathrsfs,graphicx,tikz,array}
\usepackage{subcaption}
\usepackage{hyperref}

\usepackage{xpatch}
\makeatletter
\xpatchcmd{\@thm}{\thm@headpunct{.}}{\thm@headpunct{}}{}{}
\makeatother
\usepackage{enumerate}
\usepackage[shortlabels]{enumitem}
\usepackage{color}

\newcolumntype{P}[1]{>{\centering\arraybackslash}p{#1}}

\newtheorem{theorem}{Theorem}[section]

\newtheorem{lemma}[theorem]{Lemma}

\theoremstyle{definition}

\newtheorem*{ThmA}{Theorem A}

\newtheorem*{CorB}{Corollary B}

\newcommand{\irr}{{\mathrm {Irr}}}

\newcommand{\cd}{{\mathrm {cd}}}
\newcommand{\cod}{{\mathrm {cod}}}

\newcommand{\syl}{\mathrm{Syl}}

\newcommand{\sg}{|G|}

\newcommand{\I}{\mathrm{I}}

\makeatletter
\newcommand*{\rom}[1]{\expandafter\@slowromancap\romannumeral #1@}
\makeatother

\begin{document}
	
	\title{Finite Groups With coprime non-linear codegrees}
	
	\author[A. Zarezadeh]{Ashkan ZareZadeh$^{1}$}
	\author[B. Khosravi]{Behrooz Khosravi$^1$}
	\author[Z. Akhlaghi]{Zeinab Akhlaghi$^{1,2}$}

	\address{$^{1}$ Faculty of Mathematics and Computer Science, Amirkabir University of Technology (Tehran Polytechnic), 15914 Tehran, Iran.}
	\address{$^{2}$ School of Mathematics,
		Institute for Research in Fundamental Science (IPM)
		P.O. Box:19395-5746, Tehran, Iran. }
	\email{\newline \text{(Z. Akhlaghi) }z\_akhlaghi@aut.ac.ir  \newline \text{(B. Khosravi) }khosravibbb@yahoo.com \newline \text{(A. Zarezadeh) }ashkanzarezadeh@gmail.com}

	\begin{abstract}
		 Given a finite group $G$ with an irreducible character $\chi\in\irr(G)$, the codegree of $\chi$ is defined by $\cod(\chi):=|G:\ker\chi|/\chi(1)$. The set of non-linear irreducible character codegrees of $G$ is denoted by $\cod(G|G')$. In this note, we classify all finite groups $G$ with $|\cod(G|G')|>1$ and  for each pair of distinct elements $m , n\in \cod(G|G')$, $m$ and $n$ are coprime. 
		
	\end{abstract}
	\thanks{The third author is supported by a Grant from IPM (no. 1402200112 ) }
	
	\maketitle
	
	\section{Introduction}
	Throughout this note,   $G $  is a finite group and  $ \irr(G) $  is the set of all irreducible complex characters of       $ G $, and $\pi(G)$ is the set of all prime divisors of $|G|$. Recall that $G$ is said to be a $2$-\textit{Frobenius group}, if $G$ has a normal series $1<K<H<G$, such that $G/K$ is a Frobenius group with $H/K$ as the Frobenius kernel, and $K$ is the Frobenius kernel of the Frobenius group $H$.
	
	In \cite{Qian},     the codegree of  $ \chi\in\irr(G) $   is defined as     $ \cod(\chi) = |G: \ker\chi|/\chi(1) $. Let $N$ be a noraml subgroup of $G$. The set of irreducible characters of $G$ whose kernel does not contain $N$ is denoted by $\irr(G|N)$, i.e. $\irr(G|N)=\irr(G)\setminus\irr(G/N)$. Therefore, $\irr(G|G')$ is the set of non-linear irreducible characters of $G$.
	Let  $\cd(G|N)=\{\chi(1)\mid \chi\in\irr(G|N)\}$, and $\cod(G|N)=\{\cod(\chi)\mid \chi\in\irr(G|N)\}$. 
	
	In \cite{Qian}, the authors studied the graph $\Gamma(G|N)$, whose vertices are the prime divisors of the elements in $\cod(G|N)$, and two distinct primes $p$ and $q$ are adjacent if there exists some $\chi\in\irr(G|N)$, such that $pq\mid \cod(\chi)$. For simplicity, $\Gamma(G|G)$ is denoted by $\Gamma(G)$.
	They showed that $\Gamma(G)$ is disconnected if and only if $G$ is either a Frobenius group or a $2$-Frobenius group. It is natural to ask \textit{if the disconnectedness of $\Gamma(G|G')$ results in the disconnectedness of $\Gamma(G)$}. In general, the answer is negative and using GAP\cite{GAP}, for $G=\text{SmallGroup}(480, 1188)$, we see that $\cod(G|G')=\{5,15,32\}$     and $\cod(G)=\{2,3,5,6,15,32\}$, which yields that $\Gamma(G|G')$ is disconnected while $\Gamma(G)$ is connected. In this note, we show that if $\Gamma(G|G')$ is disconnected and $|\cod(G|G')|=2$, then $\Gamma(G)$ is disconnected. So, the answer to the above question is true when $|\cod(G|G')|=2$.  
	
	In \cite{Herzog}, non-nilpotent  finite groups with $|\cd(G|G')|=1$ were determined. Analogously,  in \cite{QianF}, non-nilpotent groups with $|\cod(G|G')|=1$ are classified. So, it is an interesting topic to liken the known results on degrees to codegrees. In \cite{T}, the finite groups whose non-linear degrees are coprime and $|\cd(G|G')|=2$ were studied. We will see that our results classify all the finite groups whose non-linear codegrees are coprime and $|\cod(G|G')|=2$.
	
	Reckon that  $|\cod(G|G')|=2$ and $\Gamma(G|G')$ is disconnected if and only if $\cod(G|G')$ consists of two coprime integers. Obviously,  if $|\cod(G|G')|=m>1$ consists of  mutually coprime integers, then $\Gamma(G|G')$ has exactly $m$ components, and by \cite[Theorem B]{Qian}, we get that $m=2$. So, we prove the following general theorem:

	\begin{ThmA}
		Let $G$ be a finite group. If $|\cod(G|G')|>1$ and $\cod(G|G')$ consists of mutually coprime integers, then one of the following holds:
		\begin{itemize}
			\item $G$ is a Frobenius group and $G\cong C_p^k\rtimes Q_8$, for some prime $p$, and integer $k$:
			\item $G$ is a $2$-Frobenius group and if $1<K<H<G$ is a normal series such that $G/K=H/K\rtimes R$ and $H=K\rtimes J$ are  Frobenius groups with  Frobenius complements $R\leq G/K$ and    $J\leq H$, respectively, then:
			 \begin{enumerate}
				\item	$R$ is cyclic;
				\item $J\cong C_p$, for some prime $p$;
				\item $K$ is the unique minimal normal subgroup of $G$;
				\item For each $1\not=x\in K$, $C_G(x)/K\cong R_0$, where $R_0\leq R$ and $\pi(R_0)=\pi(R)$.
			\end{enumerate}
		\end{itemize}
	\end{ThmA}

If $N\unlhd G$ and $\theta\in {\rm Irr}(N)$, then  $\I_G(\theta)$ denotes the inertia subgroup of $\theta$ in $G$, and ${\rm Irr}(\theta^G)$ denotes the set of all irreducible constituents of $\theta^G$.  For natural number $n=p^km$, where $p$ is a prime, $k$ and $m$ are integers such that $(p,m)=1$, define $n_p=p^k$ and $n_{p'}=m$. For the rest of the notations, we follow \cite{Isaacs}.

	\section{preliminary results}
	Let $G$ be a finite group and $N\trianglelefteq G$ be an elementary abelian $p$-group. We know that the action of $G$ on $\irr(N)$ is equivalent to the action of $G$ on $N$.  Accordingly  $\{C_G(x)|x\in N\}=\{\I_G(\lambda)|\lambda\in\irr(N)\}$, so for each $x\in N$, there exists $\lambda_x\in\irr(N)$, such that $C_G(x)=\I_G(\lambda_x)$. We use this fact frequently and without any further reference.
	
	The following lemma is a corollary of Theorem A and the discussion following  it in \cite{QianF}.
	\begin{lemma}\label{Qian}
		Let $G$ be a finite group and $|\cod(G|G')|=1$. Then $G$ is either a $p$-group, for some prime $p$, or $G=G'\rtimes R$ is a Frobenius group with $G'\cong C_q^k$, for some prime $q$, as the Frobenius kernel and $R$ is cyclic, also $\cod(G|G')=\{q^k\}$, for some integer $k$ and  $q^k$ is cardinal of every minimal normal subgroup of $G$.
	\end{lemma}
	
	\begin{theorem}\cite[Theorem A]{Qian}\label{Qian1}
		Let $G$ be a non-abelian finite group and $p\in\pi(G)$. If $p$ does not divide any element in $\cod(G|G')$, then $P\in\syl_p(G)$ acts Frobeniusly on $G'$.
	\end{theorem}
	
	\begin{lemma}\cite[Lemma 2.1]{Qian}\label{subnormality}
		Let $G$ be a finite group and $\chi\in\irr(G)$. If $M$ is subnormal in $G$ and $\psi\in\irr(\chi_M)$, then $\cod(\psi)$ divides $\cod(\chi)$.
	\end{lemma}
	
	\begin{lemma}\cite[Lemma 2.7]{KLG}\label{2.7}
		Let $K$ be a normal abelian subgroup of $G$ and $\chi\in \irr(G)$. If $\ker(\chi)\cap K=1$, then $|K|$ divides $\cod(\chi)$.
	\end{lemma}
	
Let $G$ be a finite group acting on a module $M$ over a finite field, and $q$ be a prime
divisor of the order of $G/C_G(M)$. We say that the pair$ (G, M)$ satisfies $N_q$ if, for every
$v \in M\setminus \{0\},$ $C_G(v)$ contains a Sylow $q$-subgroup of G as a normal subgroup.

\begin{theorem}\cite[Proposition 8]{carlo}\label{sak}
	If $(G,M) $ satisfies $N_q$, then $(|M|-1)/(|C_M(Q)|-1) = n_q(G)$, where $n_q(G)$ is the number of Sylow $q$-subgroups of $G$ and $Q\in \syl_q(G)$. 
\end{theorem}

\begin{theorem}\cite[Theorem B]{Isaacs2}\label{Solve}
		 Let $N$ be a normal subgroup of a group $G$ and suppose that $|\cd(G|N)|\le 2$, then $G$ is solvable.
\end{theorem}

\begin{lemma}\label{*-group}
	Let $G$ be a finite group. If $|\cod(G|G')|\le 2$, then $G$ is solvable.
\end{lemma}
\begin{proof}
	On the contrary, assume that $G$ is a counterexample of minimal order. If $G$ is  simple, then $|\cd(G)|=|\cod(G|G')|+1\le 3$, which contradicts \cite[Theorems 12.6 and 12.15]{Isaacs}. Hence, $G$ is not simple. Notice that  by the fact that $G$  is a minimal counterexample, $G$ has a unique minimal normal subgroup, say $K$, which is non-solvable. Observe that $K\le G'$, and so $\irr(G|K)\subseteq \irr(G|G')$ and each character in $\irr(G|K)$ is faithful, and so $|\cd(G|K)|=|\cod(G|K)|\le |\cod(G|G')|\le 2$. Now Theorem \ref{Solve} shows that $G$ is solvable, which is a contradiction. 
\end{proof}
The proof of the previous lemma is a mimic of the proof of \cite[Theorem 2.3]{Zeng}, and it also shows that $G$ is solvable, if $|\cod(G|G')|\le 2$.

By Lemma \ref{*-group}, every *-group is solvable.\\

\begin{lemma}\label{nil}
	Let $G=C_p\times L$, where $L$ is a non-abelian $p'$-group, for some prime $p$. Then there exists $\chi, \psi \in \irr(G|G')$ such that $\cod(\psi)=p\cdot \cod(\chi)$, and so $(\cod(\chi),\cod(\psi))\ne 1$.
\end{lemma}
\begin{proof}
	Assume that $\chi\in \irr(L|L')$ and $\phi\in \irr(C_p|C_p)$. Let $\psi=\phi\times \chi$. Obviously $\cod(\chi)$ and $\cod(\psi)$ belong to $\cod(G|G')$. We can see that $\cod(\psi)=\frac{|G:1\times \ker\chi |}{\psi(1)}=\frac{p|L:\ker\chi|}{\chi(1)}=p\cdot\cod(\chi)$.
\end{proof}

\begin{lemma}\label{nilp}
	Each *-group is non-nilpotent.
\end{lemma}
\begin{proof}
	On the contrary, assume that $G$ is a nilpotent *-group.  Clearly, $|\pi(G)|\geq 2$. Let $\{p,q\}\subseteq \pi(G)$ and  $Q\in\syl_q(G)$ be a non-abelian Sylow $q$-subgroup of $G$.	
	 We know that $\cod(C_p\times Q)\subseteq \cod(G)$, and by Lemma \ref{nil}, we get that $C_p\times Q$ is not a *-group, implying 
	  that $G$ is not a *-group, a contradiction. 
\end{proof}

\begin{lemma}\label{Fro}
	Let $G=N\rtimes H$ be a Frobenius group with Frobenius kernel $N$. If $|\pi(N)|>1$, then $G$ is not a *-group.
\end{lemma}
\begin{proof}
		Assume that $N=P\times T$, where $P\in \syl_p(G)$ and $T$ is the normal $p'$-Hall of $N$. Let $\psi\in\irr(P/P')$, $\phi\in\irr(T/T')$ and $\chi=\psi\times \phi$. Obviously, $\cod(\psi^G)=\frac{|G:\ker\psi^G|}{|G:N|}=|N:\ker\psi^G|$, now since $T\le \ker\psi^G$, we get that $\cod(\psi^G)$ is a power of $p$. Now since $P\nless \ker\psi^G$ and $T\nless \ker\chi^G$, we get that $\cod(\chi^G)=|N:\ker\chi^G|$ has a prime divisor other than $p$, a contradiction.
\end{proof}

As a quick note, let $G=Q_{2^k}$ be a generalized quaternion group of order $2^k$, where $k\ge 3$. We know that $\cd(G)=\{1,2\}$ and $G$ has a faithful character, so $2^{k-1}\in\cod(G|G')$. Now as ${\bf Z}(G)<G'$, $G$ has a non-faithful non-linear character and so there exists   $\alpha\in\cod(G|G')$, where $\alpha\ne 2^{k-1}$. So $|\cod(G|G')|=1$ if and only if $G\cong Q_8$. This fact is used in the following lemma.

\begin{lemma}\label{Frob}
	If $G$ is a Frobenius *-group, then $G\cong C_p^k\rtimes Q_8$, for some prime $p$ and integer $k\ge 2$.
\end{lemma}
\begin{proof}
	Let $G=N\rtimes H$, where $N$  and $H$ are  the Frobenius kernel and Frobenius complement, respectively. By Lemma \ref{Fro}, $N$ is a $p$-group, for some prime $p$. We know that for every $1_N\ne\phi\in\irr(N)$, $p$ divides $\cod(\phi^G)=|N:\ker\phi^G|/\phi(1)$. Therefore $|\cod(H|H')|=1$. Since $H$ is not a Frobenius group, by Theorem \ref{Qian}, $H$ is a $q$-group, for some prime $q$. Now since $H$ is non-cyclic, as we discussed above $H\cong Q_8$, which yields that $N$ is abelian. If $N$ is a minimal normal subgroup, then $N$ is elementary abelian $p$-group, and we are done. So we may assume that $N$ is not minimal normal.  Assume that $G$ has a unique minimal normal subgroup $M$.  Choose $\theta\in\irr(N|M)$ and $\psi\in\irr(N/M)$. Obviously, $\theta^G$ is faithful and we can see that
	$$\cod(\theta^G)=\frac{\sg}{|G:N|}=|N|\quad \text{and} \quad \cod(\psi^G)=\frac{|G:\ker\psi^G|}{|G:N|}=|N:\ker\psi^G|<|N|,$$
	which is a contradiction.	
	Hence, there exists minimal normal subgroups $M_1$ and $M_2$. By induction $N/M_1$ and $N/M_2$ are elementary abelian, which means $\Phi(N)\le M_1\cap M_2=1$ and $N$ is elementary abelian.
	\end{proof}
	
		We recall that a finite group $G$ is called a \textit{$2$-Frobenius group}, if $G$ has a normal series $1<K<H<G$, such that $G/K$ is a Frobenius group with $H/K$ as the Frobenius kernel, and $K$ is the Frobenius kernel of the Frobenius group $H$. Recall that in such a $2$-Frobenius group,  normal subgroups are either contained in $K$ or contain $H$, this fact is useful in the following Theorem.

	\begin{theorem}\label{2Fro}
		Let $G$ be a $2$-Frobenius group. Let  $1<K<H<G$ be a normal series such that $G/K=H/K\rtimes R$ and $H=K\rtimes J$ are  Frobenius groups with  Frobenius complements $R\leq G/K$ and    $J\leq H$, respectively. Then $G$ is an $*$-group if and only if  all  the following occurs: 
		\begin{enumerate}
			\item $G/H\cong R$ is cyclic. 
			\item $H/K\cong J\cong C_p$, for some prime $p$;
			 \item $K$ is the unique minimal normal subgroup of $G$;
			\item For each $1\not =x\in K$, $C_G(x)/K\cong R_0$, 	where $R_0\le R$, and  $\pi(R_0)=\pi(R)$.
		\end{enumerate}
	Moreover, in this case $\cod(G|G')=\{p,|K||R_0|\}$. 
	\end{theorem}
	\begin{proof}
		
		First assume that $G$ is a $2$-Frobenius group with all the mentioned  hypothesis. Then $G'=H$ and every non-linear character $\chi \in \irr(G)$,   either $\chi \in \irr(G/K|H/K)$ or $\chi \in \irr(G|K)$. If the former case occurs, then  $\cod(\chi)=p$. Now,  assume the later case occurs. Then $\chi=\theta^G$, where $\theta\in I_G(\lambda)$ and $\lambda\in \irr(K)$ is non-principal.  Thus $\cod(\chi)=|I_G(\lambda)|/(\theta(1)|\ker\chi|)$. Since $K$ is the only minimal normal subgroup of $G$, $\ker\chi=1$. Also $I_G(\lambda)/K=C_G(x)/K\cong R_0$, for some $1\not =x\in K$. Thus $\theta$ is an extension of $\lambda$, leading to $\cod(\chi)=|I_G(\lambda)|=|K||R_0|$. Hence $\cod(G|G')=\{p, |K||R_0|\}$ and $G$ is a *-group. 
	 
		Now, we assume that $G$ is a *-group.  Note that $|\cod((G/K)|(G/K)')|\ge 1$, and $H/K$ is isomorphic to a Frobenius complement of $H$. If $G/K$ is a *-group, then by Lemma \ref{Frob}, $H/K\cong C_p^n$, for some prime $p$ and integer $n\ge 2$, which is a contradiction, as $H/K$ is a Frobeniu. So $|\cod((G/K)|(G/K)')|= 1$ and  $G/K$ is isomorphic to a group described in Lemma \ref{Qian}, and so  $(G/K)'=H/K\cong J \cong C_p$, for some prime $p$,  $p\in\cod(G|G') $ and $G/H\cong R$ is cyclic. 	Assume that $\cod(G|G')=\{p,\beta\}$, for some integer $\beta$.
		
		 Now since  $p=|H:K|=|G'K:K| \mid |G'|$, and $H$ is a Frobenius group, we get that $G'=H$. 
		If there exists $p\ne s\in\pi(R)\setminus\pi(\beta)$, then by Theorem \ref{Qian1}, we get that $s\in\syl_s(G)$, acts Frobeniusly on $G'=H$, which is impossible, as $H$ is a Frobenius group. Whence $\pi(R)\subseteq\pi(\beta)$. Next we prove that $K$ is a $q$-group, for some prime $q$.  Note that $K$ is nilpotent. On the contrary and without loss of generality, assume that $K=S\times Q$, where $S\in\syl_s(K)$ and $Q\in\syl_q(K)$,  are minimal normal subgroups of $G$, for some distinct prime numbers $s$ and $q$. Since $\irr(G|S)\subseteq \irr(G|G')$, by Lemma \ref{2.7}, we get that $|S|$ divides $\beta$. Note that as $G/S$ is not a Frobenius group, $\cod((G/S)|(G/S)')=\cod(G|G')=\{p,\beta\}$, and so $\beta$ divides $|G/S|$, hence for $S_1\in\syl_s(R)$, $|S|\le(\beta)_s\le |S_1|$. However since $R$ acts Frobeniusly on $H/K\cong C_p$ and $H/K\cong C_p$ acts Frobeniusly  on $K$, $|S_1|<p<|S|$, a contradiction, and so  $K$ is a $q$-group.
		
		Now we aim to  prove that $K$ is a minimal normal subgroup of $G$. Let $K/L$ be a chief factor of $G$, for some $L<K$. Note that since $G/L$ is not a Frobenius group there exists $1\ne x_0\in K/L$, such that $K/L< C_{G/L}(x_0)$. Assume that $\frac{C_{G/L}(x_0)}{K/L}\cong R_0\le R$. Let $\lambda\in\irr(K/L)$ be  non-principal, such that $\I_{G/L}(\lambda)=C_{G/L}(x_0)$. As $R_0$ is cyclic, by \cite[Theorem 11.22]{Isaacs} $\lambda$ extends to $\lambda_0\in\irr(\I_{G/L}(\lambda))$. Now since $\ker(\lambda_0^{G/L})=1$, we get that $\cod(\lambda_0^{G/L})=|\I_{G/L}(\lambda)|=|K/L||R_0|=\beta$. Notice that by Lemma \ref{2.7}, $|K/L|\mid \cod(\chi)$ for all $\chi \in  \irr((G/L)|(K/L))$, and so by the same discussion, for each $1\not =x\in K/L$,  we have $\beta=|C_{G/L}(x)|=|K/L||R_0|$, which implies that $\frac{C_{G/L}(x)}{K/L}\cong R_0$, for each $1\not=x\in K/L$. Now we prove that $|R_0|_q=|R|_q$. 	Let $Q\in\syl_q(G/L)$. Choose $1\ne x\in {\bf Z}(Q)\cap K/L$, and non-principal $\lambda\in\irr(K/L)$, such  that $C_{G/L}(x)=I_{G/L}(\lambda)$. Obviously, $Q\le C_{G/L}(x)$.  Now since $\beta=|I_{G/L}(\lambda)|=|R_0||K/L|$, and $|Q|$ divides $|I_{G/L}(\lambda)|$, we get that $|R_0|_q=|R|_q$, and also $(\beta)_q=|Q|$. Let $K/T$ be another chief factor of $G$. By the discussion we just had, $\beta=|K/T||R_1|=|K/L||R_0|$, for some $R_1\le R$, and also $|R_1|_q=|R|_q$. Note that $(\beta)_q=|Q|=|K/T||R_1|_q=|K/L||R_0|_q$, and since $|R_1|_q=|R|_q=|R_0|_q$, we get that $|K/T|=|K/L|$. In other words, each chief factor $K/W$ of $G$, has order $|K/L|$. 
		Let $G_0/K\le G/K$ such that $G_0/K\cong C_p\rtimes R_0$. Since $G/K$ is a Frobenius group and $R$ is cyclic, $G_0$ is unique.
		
		 Looking at  the action of $G_0/K\cong \frac{G_0/L}{K/L}$ on $K/L$, by the above argument,  we see that $C_{G_0/K}(x)$ is a non-trivial $p'$-subgroup of $G_0/K$, for all $1\ne x\in K/L $. Observe that $C_{G_0/K}(x)$ is cyclic, hence every subgroup of $C_{G_0/K}(x)$ is normal in $C_{G_0/K}(x)$. Note that since $G_0/K$ is a Frobenius group, we get that for $s\in\pi(C_{G_0/K}(x))$, $n_s(G_0/K)=p$. Now by Theorem \ref{sak}, there exists non-trivial $K_0/L< K/L$, such that $\frac{|K/L|-1}{|K_0/L|-1}=p$.
		 
	
		Now we prove that $L=1$. 
		On the contrary, let $L/L_0$ be a chief factor of $G$. We consider the following two cases:
		
		$\bullet$ Case 1) Assume that for each non-principal $\lambda\in \irr(L/L_0)$ and $\chi\in \irr(\lambda^{G/L_0})$, $\ker(\chi)=L_0$.\\ 
		Let $Q\in\syl_q(G/L_0)$. Notice that since $|R|_q=|R_0|_q$, we get that $|Q|=|K/L_0||R_0|_q$. Now let $1\ne x\in {\bf Z}(Q) \cap L/L_0$.
		Choose non-principal $\lambda_2\in\irr(L/L_0)$, such that $\I_{G/L_0}(\lambda_2)=C_{G/L_0}(x)$, which clearly contains $Q$. Let $\chi\in\irr(\lambda_2^{G/L_0})$. Obviously, $\chi=\lambda_0^{G/L_0}$, for some $\lambda_0\in\irr(\lambda_2^{\I_{G/L_0}(\lambda_2)})$. Now since $\ker\chi=L_0$, we get that $\cod(\chi)=|\I_{G/L_0}(\lambda_2)|/\lambda_0(1)=|R_0||K/L|=\beta$. Observe that $|\I_{G/L_0}(\lambda_2)|_q=|L/L_0||K/L||R_0|_q$, and so $\lambda_0(1)_q=|L/L_0|$, for each $\lambda_0\in\irr(\lambda_2^{\I_{G/L_0}(\lambda_2)})$. Recall that $$|\I_{G/L_0}(\lambda_2):L/L_0|=\sum\limits_{\lambda_0\in\irr(\lambda_2^{\I_{G/L_0}(\lambda_2)})}\lambda_0^2(1),$$ and so $|\lambda_0(1)|_q^2=|L/L_0|^2$ divides $|\I_{G/L_0}(\lambda_2):L/L_0|_q$, hence $|L/L_0|^3\mid |\I_{G/L_0}(\lambda_2)|_q$. On the other hand, since $H$ is a Frobenius group $H/K\cong C_p$ acts Frobeniusly on $L/L_0$, whence $p\mid |L/L_0|-1$. 
		Assume that $|L/L_0|\ge |K/L|$. Since $G/K$ is a Frobenius group, $|R_0|_q=|R|_q\mid p-1$. Thus $|R_0|_q<|L/L_0|$, implying that $|I_{G/L_0}(\lambda_2)|_q=|L/L_0||K/L||R_0|_q<|L/L_0|^3$, which is a contradiction. 
		Hence $|L/L_0|< |K/L|$. Considering  that $\frac{|K/L|-1}{|K_0/L|-1}=p$  and $p\mid |L/L_0|-1$, we deduce   $|K/L|-1$ does not have a primitive prime divisor, and so Zsigmondy's Theorem asserts that we have the following two cases:\\
		
		 (1) $|K/L|=q^2$, where $q=2^a-1$ is a Mersenne prime, which is impossible.  \\
		
		(2) $|K/L|=2^6$. However $\frac{2^6-1}{2^2-1}=21$ and $\frac{2^6-1}{2^3-1}=3^2$ are not prime, a contradiction .
		
		$\bullet$ Case 2) Assume that there exists non-principal $\lambda\in\irr(L/L_0)$ and $\chi\in \irr(\lambda^{G/L_0})$, such that $L_0<\ker(\chi)$.\\
		Note that since $G$ is a $2$-Frobenius group and $H\nless \ker(\chi)$, we get that $\ker\chi\le K$, and so $K/\ker\chi$ is a chief factor of $G$, hence $|K/\ker\chi|=|K/L|$. 
		Obviously, $\ker\chi\cap L=L_0$, and so $K/\ker\chi=L\ker\chi/\ker\chi\cong L/(L\cap \ker\chi)=L/L_0$. Thus,  $|K/L|=|K/\ker\chi|=|L/L_0|$.
		
		Let $\nu\in\irr(L/L_0)$ be non-principal, $\nu_1=\nu\times 1 \in\irr(L/L_0\times \ker\chi/L_0)=\irr(K/L_0)$, and $\mu\in\irr(\nu_1)$. Obviously, $\cod(\mu)=|K/L||R_0|$. On the other hand, $\mu=\theta^{G/L_0}$, for some linear $\theta\in\irr((\nu_1)^{\I_{G/L_0}(\nu_1)})$, and so $\cod(\theta^{G/L_0})=|\I_{G/L_0}(\nu_1):\ker\theta^{G/L_0}|=|\I_{G/L_0}(\nu_1)|/|\ker\chi:L_0|=|K/L||R_0|$. We get that $|\I_{G/L_0}(\nu)|=|\I_{G/L_0}(\nu_1)|=|K/L|^2|R_0|=|K/L_0||R_0|$, and so we have that $|\I_{G/L_0}(\lambda):K/L_0|=|R_0|$, for all non-principal $\lambda\in \irr(L/L_0)$. With the same argument for each non-principal $\zeta\in\irr(\ker\chi/L_0)$, we have that $|\I_{G/L_0}(\zeta):K/L_0|=|R_0|$.  	
		Now choose non-principals $\nu\in \irr(L/L_0)$ and $\zeta\in \irr(\ker\chi/L_0)$, and let $\omega=\nu\times \zeta$. Note that $\frac{\I_{G/L_0}(\omega)}{K/L_0}=\frac{\I_{G/L_0}(\nu)\cap \I_{G/L_0}(\zeta)}{K/L_0}=R_0^g$, for some $g\in G$ or $\frac{\I_{G/L_0}(\omega)}{K/L_0}=1$, as the intersection of  distinct conjugates of $R_0$ is trivial. 
		
		Assume that $\I_{G/L_0}(\omega)=K/L_0$, and so $\gamma=\omega^{G/L_0}\in \irr(\omega^{G/L_0})$. We have $\cod(\gamma)=$ $|K/L_0|/|\ker\gamma/L_0|$ $=|K/\ker\gamma|=|R_0||K/L|$, which implies that $|K/L|<|K/\ker\gamma|$, and so $K/\ker\gamma$ is not a chief factor of $G$. Now since $L_0\le \ker\gamma<K$, we get that $\ker\gamma=L_0$.
		Hence,  $|L/L_0|=|K/L|=|R_0|$. Since $|R_0|\mid (p-1)$ and $p\mid (|K/L|-1)$, we get a contradiction.  Therefore, for each non-principal $\nu\in \irr(L/L_0)$ and non-principal $\zeta\in \irr(\ker\chi/L_0)$, we have that
		$\frac{\I_{G/L_0}(\nu)}{K/L_0}=\frac{\I_{G/L_0}(\zeta)}{K/L_0}=R_0$,  which means for all $1\ne x\in K/L_0$, $\frac{C_{G/L_0}(x)}{K/L_0}=R_0^g$, for some $g\in G$, and so $R_0^g\le \frac{C_{G/L_0}(K/L_0)}{K/L_0}\trianglelefteq \frac{G/L_0}{K/L_0}$.
		Now since $\frac{G/L_0}{K/L_0}\cong G/K$, which is a Frobenius group, we get that $\frac{H/L_0}{K/L_0}\le \frac{C_{G/L_0}(K/L_0)}{K/L_0}$,  a contradiction.

		 Now since both cases lead to contradiction, we get that $L=1$ and $K$ is a minimal normal subgroup of $G$, as desired, and also $\beta=|K||R_0|$. Now since $\pi(R)\subseteq \pi(|K||R_0|)=\{q\}\cup \pi(R_0)$, and $|R_0|_q=|R|_q$, we get that $\pi(R)=\pi(R_0)$. 
\end{proof}

	\begin{lemma}\label{TH}
		Let $G=P\rtimes R$, where $P\in\syl_p(G)$ and $R$ is a Hall $p'$-subgroup of $G$, for some prime $p$. Let $N<P$ be a minimal normal subgroup of $G$ such that $G/N$ is a Frobenius group, with Frobenius kernel $P/N$. If $N$ is not a direct factor of $P$, then:
			\begin{enumerate}
				\item for each $\chi\in\irr(G|N)$, $(|\ker\chi|,|R|)=1$,
				\item $G$ is a Frobenius group or there exists $\chi\in\irr(G|N)$ such that $(\cod(\chi),|R|)\ne 1$.
			\end{enumerate}
	\end{lemma}
	\begin{proof} (1) 
		Assume that $\chi\in\irr(G|N)$. We know that $\ker\chi=P_0\rtimes R_0$, where $P_0<P$ and $R_0\le R$. If $R_0\ne 1$, then since $G/N$ is a Frobenius group and $N\ker\chi/N$ is not a $p$-subgroup of $G/N$, we have that $P/N\le N\ker\chi/N$ and we conclude that $P\le NP_0=N\times P_0$, which is a contradiction, and so $R_0=1$. \\
		(2) Assume that $(\cod(\chi),|R|)= 1$, for each $\chi\in\irr(G|N)$. 
			Let $\lambda\in\irr(\chi_N)$ for some $\chi\in \irr(G|N)$. We know that $\chi=\theta^G$ for some $\theta\in\irr(\lambda^{\I_G(\lambda)})$. Thus $\cod(\chi)=\frac{|G:\ker\chi|}{|G:\I_{G}(\lambda)|\theta(1)}=\frac{|\I_G(\lambda):\ker\theta^G|}{\theta(1)}$ and so by part (1), $|\I_G(\lambda)|_{p'}=(\theta(1))_{p'}$. Let $k=(\theta(1))_{p'}$. Recall that $|\I_G(\lambda):N|=\sum\limits_{\theta\in\irr(\lambda^{\I_G(\lambda)})}\theta(1)^2$, which is divisible by $k^2$, thus $(\theta(1))_{p'}=k=1$, and so $|\I_G(\lambda)|_{p'}=1$, which means $\I_G(\lambda)\le P$, for each $1_N\ne \lambda\in \irr(N)$. Thus,  for each $1\ne x\in N$, we have that $C_R(x)=1$, and so $R$ acts Frobeniusly on $N$ and  since $R$ acts Frobeniusly on $P/N$, we conclude that $G$ is a Frobenius group.
	\end{proof}

	\begin{lemma}\label{1p}
		Let $G$ be a *-group and  $N$ a minimal normal $q$-subgroup, for some prime $q$. If $G/N$ is a $p$-group, for some prime $p$, then $G$ is a Frobenius group.
	\end{lemma}
	\begin{proof}
		Obviously, $G=N\rtimes P$, where $P\in\syl_p(G)$ and $\cod(G|G')=\{p^m,q^k\}$, for some $m$ and $k$. If $C_N(P)=N$, then $G=N\times P$, which contradicts Lemma \ref{nilp}. Hence, as $N$ is  minimal normal,  we get that $C_N(P)=1$, and so for each $1\ne x\in N$, there exists a $p$-subgroup $P_x\notin\syl_p(G)$ such that $ C_G(x)=N\rtimes P_x< G$.
		Now we prove that for each $1\ne x\in N$, $C_G(x)\le N$. On the contrary, assume that there exists $1\ne z\in N$ such that $P_z\ne 1$. Choose $\lambda\in \irr(N)$ such that $C_G(z)=\I_G(\lambda)$. By \cite[Problem 6.18]{Isaacs}, we get that $\lambda$ extends to $\lambda_0\in\irr(\I_G(\lambda))$. Note that $\ker\lambda_0^G\le P_z $ and by Lemma \ref{2.7}, $|N|$ divides $\cod(\lambda_0^G)=|\I_G(\lambda):\ker\lambda_0^G|$. As $\cod(\lambda_0^G)=q^k$ we conclude that $\ker\lambda_0^G=P_z$ and so $\I_G(\lambda)=N\times P_z$. Now choose $\theta\in\irr(P_z/P_z')$. We observe that $\lambda$ extends to $\chi=\lambda\times\theta$. Now it is easy to see that $pq\mid |\I_G(\lambda):\ker\chi^G|=\cod(\chi^G)$, a contradiction. So $C_G(x)=N$ for each $1\ne x\in N$, which means $P$ acts Frobeniusly on $N$.
	\end{proof}

\begin{lemma}\label{1Q}
	Let $G$ be a *-group and $N$ be a minimal normal $q$-subgroup of $G$ such that $\overline{G}=G/N$ is non-nilpotent and $|\cod(\overline{G}|\overline{G}')|=1$. Then $G$ is a $2$-Frobenius group.
\end{lemma}
\begin{proof}
	By Lemma \ref{Qian}, $G/N=K/N\rtimes R$ is a Frobenius group, where $K/N=(G/N)'$ is an elementary abelian $p$-group and $R$ is a cyclic Hall $p'$-subgroup of $G$. Note that $p^k\in\cod(\overline{G}|\overline{G}')\subset \cod(G|G')$, for some $k$. We study the following two cases.
	
	\bigskip
	Case 1) $p=q$.\\
	In this case, $NG'\in \syl_p(G)$ is normal in $G$. 
	 We claim that $N$ is not a direct factor of $G'N$. Otherwise, $G'N$ is an elementary abelian $p$-group and so by Maschke Theorem there exists $L\le G'$ such that $G'N=N\times L$ and  $L\trianglelefteq G$. We may assume that $L$ is a minimal normal subgroup of $G$. Note that if $G$ is a Frobenius, then by Lemma \ref{Frob}, $R\cong Q_8$, which is a contradiction. Hence we get that $R$ does not act Frobeniusly on $N$, thus $N$ and $L$ are not $G$-isomorphic and so $N$ and $L$ are the only minimal normal subgroups of $G$. We choose $1_N\ne\psi\in\irr(N)$ and $1_L\ne\phi\in\irr(L)$. Let $\theta=\psi\times\phi$ and $\chi\in\irr(\theta^G)$. Note that since $N\times L=G'N\nleq \ker\chi$, we get that $\chi$ is faithful and non-linear. Now by Lemma \ref{2.7}. $|G'N|\mid \cod(\chi)$, and since $p^k<|G'N|$, we get a  contradiction. Therefore $N$ is not a direct factor of $G'N$.
	Now by Lemma \ref{TH}(2), there exists $\theta \in\irr(G|N)\subset \irr(G|G'N)$ such that $(\cod(\theta),|R|)\ne 1$. So $\cod(G|G'N)=\{p^k,\cod(\theta)\}$, however by Lemma \ref{2.7}, $p \mid \cod(\theta)$, which is a contradiction.
	
	\bigskip
	Case 2) $p\ne q$.\\
	In this case, $K=N\rtimes L$, for some $L\in\syl_p(K)$.
	Our goal is to show that $L$ acts Frobeniusly on $N$, which shows that $G$ is a $2$-Frobenius group. First of all since $K'\le N$, we get that $K'=1$ or $K'=N$. Assume that $K$ is abelian. Hence $K=N\times L$. Note that for each $1\ne x\in L$, $\overline{C_G(x)}=C_{\overline{G}}(\overline{x})=\overline{K}$, and so $C_G(x)=K$, henceforth $\I_G(\lambda)=K$, for each $1_L\ne \lambda\in\irr(L)$. Let $1_N\ne \lambda_0\in\irr(N)$ and $\chi=\lambda_0\times \lambda$. Then $\chi\in\irr(K)$ and $\chi_L=\lambda$. Now as $pq\mid \cod(\chi)$, by Lemma \ref{subnormality}, we conclude that $pq\mid \cod(\chi^G)=|K:\ker\chi^G|$, a contradiction. So $K'=N$. We claim that $p$ does not divide any element in $\cod(K|K')$. On the contrary, assume that $p\mid \cod(\theta)$, for some $\theta\in\irr(K|K')$. Note that $q\mid |K:\ker\theta|$ and by Ito's Theorem, $\theta(1)$ divides $|K:N|=L$, which means $pq\mid \frac{|K|}{\theta(1)|\ker\theta|}=\cod(\theta)$ and it is a contradiction by Lemma \ref{2.7} and so $p$ does not divide any element in $\cod(K|K')$. Now Theorem \ref{Qian1} asserts that $L$ acts Frobeniuosly on $K'=N$ and $K$ is a Frobenius group, as we wanted.
\end{proof}

	\begin{lemma}\label{2N}
		Let $G$ be a *-group and $N$ be a minimal normal $q$-subgroup of $G$, for some prime $q$. If $\overline{G}=G/N\cong C_p^n\rtimes Q_8$, for some odd prime $p$ and some integer $n>1$, and $\overline{G}$ is a Frobenius group, then $N$ is a $p$-group.
	\end{lemma}
	\begin{proof} Obviously, $n>1$.
		Note that since $\{4,p^k\}\subseteq \cod(\overline{G}|\overline{G}')\subseteq \cod(G|G')$, for some $k$, we get that $\cod(G|G')=\{4,p^k\}$.
		On the contrary, assume  $N$ is not a $p$-group. 
		If there exists $1_N\ne \gamma\in\irr(N)$ such that $\gamma$ extends to $\gamma_0\in \irr(G)$, then we choose $\phi\in\irr(Q_8)$, such that $\phi(1)=2$. Note that $\gamma_0\phi\in\irr(G)$ is non-linear and $N\cap\ker(\gamma_0\phi)=1$, so by Lemma \ref{2.7}, $|N|$ divides some elements in $\cod(G|G')=\{4,p^k\}$. Also in the other case, when there  exists $1_N\ne \gamma\in\irr(N)$, such that $\gamma^G$ has a non-linear constituent, again by Lemma \ref{2.7}, $|N|$ divides some element in $\cod(G|G')=\{4,p^k\}$, hence
		$q=2$ and $|N|\le 4$ or $q=p$. Assume that $q=2$.  Let $C_p^n\cong P\in\syl_p(G)$. Note that since $G/N$ is a Frobenius group, we get that for each $1\ne x\in P$,
		$C_G(x)\le H$, where $H=NP=N\rtimes P$. We claim that $P$ is not normal in $H$. Otherwise, $H=N\times P$, and $C_G(x)=H$ for each $1\ne x\in P$. Choose $1_N\ne\lambda_0\in\irr(N)$, and $\lambda_x\in\irr({P})$ such that $\I_G(\lambda_x)=H$. Obviously, $\lambda_x$ extends to $\chi =\lambda_x\times \lambda_0$. Observe that $2p\mid |H:\ker\chi|$, therefore $2p\mid |H:\ker\chi^G|=\cod(\chi^G)$, a contradiction. Hence $P$ is not normal in $G$, and so there exists $y\in N$, such that $C_H(y)<H$,
		 which means $N\cap {\bf Z}(H)=1$ and we get that $|N|=4$. Therefore $N \cong C_2\times C_2$, and $H/C_H(N)\le S_3$. Note that since $2\nmid |H:C_H(N)|$, we get that $|H:C_H(N)|=3$, which means $p=3$ and $C_H(N)$ is a maximal subgroup of $H$, therefore $C_H(z)=C_H(N)=N\times C_3^{n-1}$, for each $1\ne z\in N$. Now choose a non-principal character $\omega\in\irr(C_3^{n-1})$ and $1_N\ne \lambda\in\irr(N)$. Obviously, $\lambda$ extends $\phi=\lambda\times \omega$ and $6\mid |H:\ker\phi^H|=\cod(\phi^H)$. As $\lambda$ does not extend to $G$, there exists a non-linear irreducible character in $\irr(\phi^G)$ and by Lemma \ref{subnormality}, $6=2p$ divides some elements in $\cod(G|G')$, which is a contradiction and we conclude that $q=p$, and so $N$ is a $p$-group.
	\end{proof}

\begin{lemma}\label{2q}
	Let $G$ be a finite group and $N$ be a minimal normal $q$-subgroup of $G$, for some prime $q$. If $G/N$ is a  $2$-Frobenius group, then $G$ is not a *-group.
\end{lemma}
\begin{proof}
	On the contrary, assume that $G$ is a *-group. Let $1<K/N<H/N<G/N$ be a normal series such that $G/K$ is a Frobenius group with kernel $H/K$ and $H/N$ is a Frobenius group with kernel $K/N$. By Lemma \ref{2Fro}, $G/K\cong H/K\rtimes R$, is a Frobenius group for some cyclic subgroup $R$, and $K/N$ is a minimal normal $s$-subgroup of $G/N$, for some prime $s$. Also $(G/N)'=H/N\cong K/N \rtimes C_p$, for some prime $p$, and $p\in\cod(G|G')$. Moreover, $\cod(G/N|(G/N)')=\{p, |K/N||R_0|\}$, where $R_0< R$. Therefore, $\cod(G|G')=\{p, |K/N||R_0|\}$. 
	We get a contradiction in the following two cases:
	
	\bigskip
	Case 1) Let $p=q$.\\
	In this case, $K=N\rtimes L$, for some $L\in \syl_s(K)$. We claim that $N\le {\bf Z}(K)$. Otherwise, $N\cap {\bf Z}(K)=1$. So there exists $1_N\ne\lambda\in\irr(N)$  such that $\I_H(\lambda)<H$, which means that $\lambda$ is not extendable to $G$, and every $\chi\in\irr(\lambda^G)$ is non-linear.  If $|N|\ne p$, then as $N\cap \ker\chi=1$, by Lemma \ref{2.7} we get that $|N|\mid\cod(\chi)$. So $p^2\mid \cod(\chi)\in\cod(G|G')$, a contradiction. Hence $N=\langle x\rangle \cong C_p$. Observe that $p\mid (|L|-1)$, therefore $L$ does not act Frobeniusly on $N$, which means $N< C_K(x)=C_K(N)<K$. Note that $C_K(x)=N\times L_1$, for some $s$-subgroup $L_1\notin \syl_s(K)$. Let $1_{L_1}\ne\theta\in\irr(L_1)$. Obviously, $\psi=\lambda\times\theta$ does not extend to $G$ and $ps\mid\cod(\psi)$. Also since $C_K(N)$ is subnormal in $G$, by Lemma \ref{subnormality}, $ps$ divides an element in $\cod(G|G')$, a contradiction. So $N\le {\bf Z}(K)$ and $K=N\times L$. Note that $C_H(l)=K$, for each $1\ne l\in L$. We choose $1_L\ne\omega\in\irr(L)$ such that $\I_H(\omega)=K$. Now similar to the above discussion about $\lambda$, we get a contradiction.

	\bigskip
	Case 2) Let $p\ne q$\\
	In this case, $H=K\rtimes P$, for some $C_p\cong P\in\syl_p(H)$. 
	If $P$ acts Frobeniusly on $K$, then $G$ is a $2$-Frobenius *-group, and so by Lemma \ref{2Fro}, $K$ is a minimal normal subgroup of $G$,  a contradiction.
	Hence there exists $1_N\ne \lambda\in\irr(N)$, such that $P\le \I_G(\lambda)$. We claim that $\lambda$ does not extend to $G$. Otherwise, there exists a linear character $\lambda_0\in\irr(\lambda^G)$ and $G'\cap N=1$. Recall that $G'N=H$, and so $H=G'\times N$ and $P\le G'\le\ker\lambda_0$. Choose non-principle $\psi\in\irr(H/K)$, and by Gallagher's  Theorem let $\chi=\lambda_0\psi^G\in\irr(G)$. Note that $N\nless \ker\chi$ and $P\nless \ker\chi$, which implies that $p\nmid |\ker\chi|$. Also $\chi(1)=|R|$, hence $p\nmid \chi(1)|\ker\chi|$ which implies that $p|N|\mid \cod(\chi)$, a contradiction. Therefore our claim holds and every constituent of $\lambda^G$ is non-linear. If there exists $\theta\in\irr(\lambda^G)$ such that $p\mid \cod(\theta)$, then by Lemma \ref{2.7}, $p|N|\mid\cod(\psi)$, which is a contradiction. Hence for each $\theta\in\irr(\lambda^G)$, $p\nmid \cod(\theta)$. If $p\mid |\ker\theta|$, then $\ker\theta\cap H\ne 1$, since every Sylow $p$-subgroup of $G$ is contained in $H$. Let $M:=\ker\theta\cap H$. Note that since $MN/N\trianglelefteq H/N$, $H/N$ is a Frobenius group and also $p\mid |MN:N|$, we get that $H=MN$. Observe that $N\cap M=N\cap\ker\theta\cap H=1$, so $H=M\times N\cong (K/N\rtimes P)\times N$. Let $1_P\ne\lambda_1\in\irr(P)$ and $\omega=\lambda_1\times\lambda$, so $pq\mid \cod(\omega)$. Note that $\omega$ does not extend to $G$, hence $\omega^G$ has a non-linear constituent $\chi$, and so by Lemma \ref{subnormality}, $\cod(\chi)$ divides an element in $\cod(G|G')$, a contradiction. Hence for each $\theta\in\irr(\lambda^G)$,  $p\nmid |\ker\theta|$ and since $p\nmid \cod(\theta)$, we get that $p\mid \theta(1)$. Remember that $\theta=\gamma^G$, for some $\gamma\in\irr(\I_G(\lambda))$. As $P\le\I_G(\lambda)$ and $p\mid \theta(1)=|G:\I_G(\lambda)|\gamma(1)$, we conclude that $p\mid \gamma(1)$. Now since $|\I_G(\lambda):N|=\sum\limits_{\gamma\in\irr(\lambda^{\I_{G}(\lambda)})}\gamma(1)^2$, we get that $p^2\mid |G|$, which is a contradiction. 
\end{proof}

Now we are ready to prove Theorem A:
\bigskip

%
{\bf {\it Proof of Theorem A.}}
 According to Theorems \ref{Frob} and \ref{2Fro}, if we prove  $G$ is  a Frobenius group or a $2$-Frobenius group, then we are done. So, on the contrary, assume that $G$ is a *-group of minimal order which is neither a  Frobenius group  nor  a $2$-Frobenius group. If $G/N$ is abelian for each minimal normal subgroup $N$ of $G$, then $G'$ is the unique minimal subgroup of $G$ and by \cite[Theorem 12.3]{Isaacs}, $G$ is either a $p$-group or a group explained in Lemma \ref{Qian}, a contradiction. So there exists a minimal normal subgroup $N$ of $G$ such that $G/N$ is non-abelian. If $|\cod((G/N)|(G/N)')|=1$, then by Lemma \ref{Qian}, $G/N$ is either a $p$-group or a Frobenius group,  which  is  contradicting  Lemmas \ref{1p} and \ref{1Q} respectively. So $|\cod((G/N)|(G/N)')|=2$, and $G/N$ is a *-group. By Lemma \ref{2q}, we get that $G/N$ is not a $2$-Frobenius group, and so by minimality of $|G|$, we get that $G/N$ is a Frobenius *-group and so by Lemma \ref{Frob}, $G/N\cong C_p^k\rtimes Q_8$, for some prime $p$ and integer $k$. Notice that for each $1\ne l\in L$, $C_P(l)\le N$.
	 By Lemma \ref{2N}, $N$ is  a $p$-group. So $G=P\rtimes L$, where $P\in\syl_p(G)$ and $Q_8\cong L\in\syl_2(G)$. Note that every minimal normal subgroup of $G$ is contained in $P$. If there exists another minimal normal subgroup $M\ne N$, then since the Sylow $p$-subgroup of $G/M$ is normal, we get that $G/M$ is not a $2$-Frobenius group and so by minimality of $|G|$, we see that $G/M$ is also a Frobenius group, which implies that for each $1_L\ne l\in L$, $C_P(l)\le N\cap M=1$ and $G$ is a Frobenius group, a contradiction. Hence $N$ is the unique minimal normal subgroup of $G$. Assume that there exists $1_N\ne\lambda\in\irr(N)$, such that $\lambda$ extends to $\theta\in\irr(G)$. Let $\chi\in\irr(Q_8)$ such that $\chi(1)=2$. Note that $N\cap \ker(\theta\chi)=1$, and so $\theta\chi$ is faithful and $\cod(\theta\chi)=|G|/2=2|P|$ which is a contradiction, since $\cod(G|G')=\cod(G/N|(G/N)')=\{4,p^a\}$, for some integer $a$. So non-principal irreducible characters of $N$ do not extend to $G$. Let $1_N\ne\lambda\in\irr(N)$. Evidently, for each $\chi\in\irr(\lambda^G)$, $N\cap\ker\chi=1$. Also $\chi=\theta^G$ for some $\theta\in\irr(\I_G(\lambda))$ and by Lemma \ref{2.7}, $|N|\mid\cod(\chi)$, so $p^a=\cod(\chi)=|G|/\chi(1)=|\I_G(\lambda)|/\theta(1)$, hence $\theta(1)=\frac{|\I_G(\lambda)|}{p^a}$. Let $k=\theta(1)$. Since 
	$|\I_G(\lambda):N|=\sum\limits_{\theta\in\irr(\lambda^{\I_G(\lambda)})}\theta(1)^2=k^2|\irr(\lambda^G)|$, we get that $2\nmid k$ and so $2\nmid |\I_G(\lambda)|$, which means for each $1\ne x\in N$, $C_G(x)\le P$, and $Q_8$ acts Frobeniusly on $N$. Recall that $Q_8$ also acts Frobeniusly on $P/N$. Whence $G$ is a Frobenius group, a contradiction.
$\blacksquare$


\begin{thebibliography}{1}
	

	
	
	
\bibitem{KLG} Akhlaghi, Z., Ebrahimi, M.  Khatami, M. On the Multiplicities of the Character Codegrees of Finite Groups. Algebr Represent Theor 26, 3085–3100 (2023). https://doi.org/10.1007/s10468-022-10183-w	

\bibitem{Herzog}  Bianchi, M., Gillio Berta Mauri, A., Herzog, M., Qian, G.,  Shi, W.,
Characterization of non-nilpotent groups with two irreducible character degrees, Journal of Algebra, Volume 284, Issue 1, 2005, Pages 326-332

\bibitem{carlo} Casolo, C. (2010). Some linear actions of finite groups with q'-orbits, J. Group Theory. 13: 503–534.
	

\bibitem{GAP} The GAP Group, GAP -- Groups, Algorithms, and Programming, Version 4.14.0; 2024. (https://www.gap-system.org)

	
\bibitem{Isaacs} Isaacs, I. M. (1976). Character Theory of Finite Groups. New York: Academic Press.

\bibitem{Isaacs2}I.M. Isaacs, Greg Knutson, Irreducible Character Degrees and Normal Subgroups, Journal of Algebra, Volume 199, Issue 1, 1998, Pages 302-326, ISSN 0021-8693,
https://doi.org/10.1006/jabr.1997.7191.

\bibitem{T} Noritzsch T.,
Groups Having Three Complex Irreducible Character Degrees, Journal of Algebra, Volume 175, Issue 3, 1995, Pages 767-798

\bibitem{Qian} Qian, G., Wang, Y., Wei, H. (2007). Codegrees of irreducible characters in finite groups. J. Algebra 312:946-955.

\bibitem{QianF} Qian, G.,  Zeng, Y. (2022). Finite groups in which all nonlinear irreducible characters have the same codegree. Communications in Algebra, 51(3), 1293–1297. https://doi.org/10.1080/00927872.2022.2134406
	
\bibitem{Zeng} Yang, D., Zeng, Y. Nonsolvable Groups with Three Nonlinear Irreducible Character Codegrees. Chin. Ann. Math. Ser. B 45, 173–178 (2024). https://doi.org/10.1007/s11401-024-0010-z
	





\bibitem{qian1} Qian, G.  (2011). {\it A note on element orders and character codegrees.} \newblock {\em Arch. Math.} 95: 99-103.

\bibitem{qian} Qian, G.  (2021). {\it Element orders and character codegrees.} \newblock {\em Bull. London Math. Soc.} 53: 820-824.


\end{thebibliography}
\end{document}